\documentclass[11pt]{amsart}

\usepackage{amsmath,amssymb,amsfonts}
\usepackage{amsthm}
\usepackage{geometry}
\usepackage{hyperref}
\usepackage{enumitem}

\DeclareMathOperator{\csch}{csch}

\newtheorem{theorem}{Theorem}[section]

\newtheorem{proposition}[theorem]{Proposition}
\newtheorem{corollary}[theorem]{Corollary}
\newtheorem{conjecture}[theorem]{Conjecture}

\theoremstyle{definition}

\theoremstyle{remark}
\newtheorem{remark}[theorem]{Remark}

\title[A Heat Kernel Expectation Approach]{
A Heat Kernel Expectation Approach to Boundary-Corrected
Li--Yau Estimates for the Dirichlet Heat Equation
}

\author{Li-Chang Hung}
\address{Department of Mathematics, Soochow University, Taipei, Taiwan}
\email{lichang.hung@gmail.com }
\thanks{
}

\begin{document}

\maketitle

\begin{abstract}

In this paper, we establish an explicit boundary-corrected
Li--Yau type gradient estimate for positive solutions of the
Dirichlet heat equation on the Euclidean half-space. The main idea is to exploit the reflection structure of the Dirichlet heat kernel and introduce a normalized kernel-induced
probability measure. Under this representation, logarithmic
derivatives of the heat kernel become expectations of explicit
kernel quantities. The reflected Gaussian component generates a hyperbolic correction
term involving

\[
\coth\left(\frac{x_n y_n}{2t}\right),
\]
which has no analogue in the whole Euclidean heat equation. Using Jensen's inequality and the sharp estimate
\[
0<z\,\csch z<1 ,
\qquad z>0,
\]
we prove that every positive solution satisfies
\[
\Delta\log w(x,t)
\geq
-\frac n{2t}
-\frac1{x_n^2}.
\]
The first term represents the classical Euclidean Li--Yau diffusion
scaling, while the second term is an explicit inverse-square
correction determined by the distance to the Dirichlet boundary.

Our approach provides a direct kernel interpretation of the
boundary effect and suggests possible extensions to more general
domains.

\end{abstract}

\section{Introduction}

The analysis of heat equations and their gradient estimates is one
of the central topics in partial differential equations, geometric
analysis, and diffusion theory. Heat kernels provide fundamental
information about propagation of diffusion, while differential
Harnack inequalities reveal quantitative relations between spatial
and temporal variations of solutions.

The classical Li--Yau inequality is one of the most influential
results in this direction. For a positive solution $u$ of

\[
u_t=\Delta u ,
\]

on a complete Riemannian manifold, Li and Yau established the
fundamental estimate

\[
\frac{|\nabla u|^2}{u^2}
-
\alpha\frac{u_t}{u}
\leq
\frac{C}{t},
\]

which gives sharp control of the space-time behavior of diffusion
processes. This result has become a basic tool in geometric analysis
and heat kernel theory; see Li--Yau~\cite{LiYau1986}.

The relationship between heat kernel estimates, spectral theory,
and diffusion processes has been extensively developed in the
classical works of Davies~\cite{Davies1989} and Grigor'yan
\cite{Grigor2009}. Moreover, Sobolev inequalities and their
connection with heat propagation have been studied systematically
in the framework of Saloff-Coste~\cite{SaloffCoste2002}.

While the classical Li--Yau inequality describes diffusion on
complete spaces without boundary, the presence of a boundary changes
the structure of the heat kernel. In particular, the Dirichlet heat
kernel contains a reflected Gaussian component corresponding to the
absorbing boundary. This additional term creates new nonlinear
effects which are absent in the whole-space case.

A pioneering analysis of this phenomenon was given by Hamilton
\cite{Hamilton2011}. For the one-dimensional Dirichlet problem

\[
f_t=f_{xx},
\qquad x>0,
\]

with

\[
f(0,t)=0,
\]

Hamilton introduced a nonlinear auxiliary quantity and proved the
sharp estimate

\[
l_{xx}
+\frac1{2t}
+
\frac1{x^2}
F^2
\left(
xl_x+\frac{x^2}{2t}
\right)
\geq0,
\]

where

\[
v=\frac{u}{\tanh u},
\qquad
F(v)=\frac{u}{\sinh u}.
\]

In particular,

\[
l_{xx}
+\frac1{2t}
+\frac1{x^2}
\geq0 .
\]

Hamilton's result demonstrates that the Dirichlet boundary produces
an inverse-square correction term. However, the proof relies on a
nonlinear maximum principle argument, and the analytic origin of the
boundary correction is not directly visible from the heat kernel
structure.

While Hamilton's estimate was obtained through a nonlinear maximum
principle argument, our approach provides a complementary kernel
representation viewpoint, showing that the boundary correction
arises intrinsically from the reflected Gaussian structure of the
Dirichlet heat kernel.

The purpose of the present paper is to provide an explicit kernel
interpretation of this correction mechanism.

We consider the Dirichlet heat equation on the Euclidean half-space

\[
\mathbb R^n_+
=
\{x=(x_1,\ldots,x_n)\in\mathbb R^n:x_n>0\}.
\]

Let $w(x,t)$ satisfy

\[
\begin{cases}
w_t=\Delta w,
&
(x,t)\in\mathbb R^n_+\times(0,\infty),
\\[1mm]
w=0,
&
x_n=0 .
\end{cases}
\]

The Dirichlet heat kernel is obtained by reflection:

\[
K_D(x,y,t)
=
\frac1{(4\pi t)^{n/2}}
\left(
e^{-\frac{|x'-y'|^2+(x_n-y_n)^2}{4t}}
-
e^{-\frac{|x'-y'|^2+(x_n+y_n)^2}{4t}}
\right),
\]

where

\[
x=(x',x_n),
\qquad
y=(y',y_n).
\]

The crucial observation is that the reflected Gaussian term
produces a hyperbolic structure in the normal direction:

\[
\frac{\partial_{x_n}K_D}{K_D}
=
\frac{y_n}{2t}
\coth
\left(
\frac{x_ny_n}{2t}
\right)
-
\frac{x_n}{2t}.
\]

This identity reveals the precise origin of the boundary correction.
Unlike the whole-space Gaussian kernel, the logarithmic derivative
contains an additional nonlinear factor generated by reflection.

Our main result is the following boundary-corrected Li--Yau
estimate:

\[
\boxed{
\Delta\log w(x,t)
\geq
-\frac n{2t}
-\frac1{x_n^2}
}.
\]

The correction term can be written geometrically as

\[
-\frac1{d(x,\partial\mathbb R^n_+)^2},
\]

where

\[
d(x,\partial\mathbb R^n_+)=x_n .
\]

Therefore the estimate separates the universal diffusion effect from
the contribution of the absorbing boundary.

The proof is based on a normalized kernel measure

\[
d\mu_{x,t}(y)
=
\frac{K_D(x,y,t)g(y)}{w(x,t)}\,dy ,
\]

which transforms logarithmic derivatives of the solution into
expectations of explicit kernel quantities. Combining this
representation with Jensen's inequality and the elementary
hyperbolic estimate

\[
0<z\csch z<1
\]

gives the desired boundary correction.

The organization of this paper is as follows. In Section 2 we
introduce the Dirichlet heat kernel representation and the associated
probability measure. Section 3 establishes the main logarithmic
derivative identities. Section 4 proves the boundary-corrected
Li--Yau estimate. Finally, Section 5 discusses interpretations and
possible extensions.

\section{Dirichlet Heat Kernel Representation}

Throughout the paper we write

\[
x=(x_1,\ldots,x_{n-1},x_n)
=(x',x_n),
\]

where

\[
x_n>0.
\]

Let $g\geq0$ be the initial datum. The solution of the Dirichlet
heat equation is represented as

\[
w(x,t)
=
\int_{\mathbb R^n_+}
K_D(x,y,t)g(y)\,dy .
\]

The Dirichlet heat kernel is

\[
K_D(x,y,t)
=
\frac1{(4\pi t)^{n/2}}
e^{-\frac{|x'-y'|^2}{4t}}
\left(
e^{-\frac{(x_n-y_n)^2}{4t}}
-
e^{-\frac{(x_n+y_n)^2}{4t}}
\right).
\]

The difference between the two Gaussian terms encodes the absorbing
boundary condition.

The following identity is the key structural formula:

\[
(x_n+y_n)^2-(x_n-y_n)^2
=
4x_ny_n .
\]

Hence the ratio between the reflected and original Gaussian terms
naturally produces the hyperbolic function

\[
\coth\left(\frac{x_ny_n}{2t}\right).
\]

\subsection{Normalized Kernel Measure}

For fixed $(x,t)$, define the probability measure

\[
d\mu_{x,t}(y)
=
\frac{
K_D(x,y,t)g(y)
}{
w(x,t)
}
\,dy .
\]

Since

\[
w(x,t)
=
\int_{\mathbb R^n_+}
K_D(x,y,t)g(y)\,dy ,
\]

we have

\[
\int_{\mathbb R^n_+}
d\mu_{x,t}(y)=1 .
\]

Therefore, for any integrable function $\Phi(y)$,

\[
\frac1{w(x,t)}
\int_{\mathbb R^n_+}
\Phi(y)
K_D(x,y,t)g(y)\,dy
=
\int_{\mathbb R^n_+}
\Phi(y)d\mu_{x,t}(y).
\]

This probability representation allows us to interpret
logarithmic derivatives of the heat kernel as expectations.

\section{Kernel Logarithmic Derivative Identities}

We first derive the identities generated by the Dirichlet heat
kernel.

\begin{proposition}[Normal logarithmic derivative identity]

For the Dirichlet heat kernel on $\mathbb R^n_+$,

\[
\frac{\partial_{x_n}K_D(x,y,t)}
{K_D(x,y,t)}
=
\frac{y_n}{2t}
\coth
\left(
\frac{x_ny_n}{2t}
\right)
-
\frac{x_n}{2t}.
\]

\end{proposition}

\begin{proof}

Only the normal part of the kernel is involved. Write

\[
K_D(x,y,t)
=
A(x',y',t)
\left(
e^{-\frac{(x_n-y_n)^2}{4t}}
-
e^{-\frac{(x_n+y_n)^2}{4t}}
\right),
\]

where

\[
A(x',y',t)
=
\frac1{(4\pi t)^{n/2}}
e^{-\frac{|x'-y'|^2}{4t}} .
\]

Differentiating with respect to $x_n$ gives

\[
\partial_{x_n}K_D
=
A
\left(
-\frac{x_n-y_n}{2t}
e^{-\frac{(x_n-y_n)^2}{4t}}
+
\frac{x_n+y_n}{2t}
e^{-\frac{(x_n+y_n)^2}{4t}}
\right).
\]

Hence

\[
\frac{\partial_{x_n}K_D}{K_D}
=
\frac{
-\frac{x_n-y_n}{2t}
e^{-\frac{(x_n-y_n)^2}{4t}}
+
\frac{x_n+y_n}{2t}
e^{-\frac{(x_n+y_n)^2}{4t}}
}
{
e^{-\frac{(x_n-y_n)^2}{4t}}
-
e^{-\frac{(x_n+y_n)^2}{4t}}
}.
\]

Multiplying numerator and denominator by

\[
e^{\frac{(x_n-y_n)^2}{4t}},
\]

we obtain

\[
=
\frac{
-\frac{x_n-y_n}{2t}
+
\frac{x_n+y_n}{2t}
e^{-\frac{x_ny_n}{t}}
}
{
1-e^{-\frac{x_ny_n}{t}}
}.
\]

After rearrangement,

\[
=
-\frac{x_n}{2t}
+
\frac{y_n}{2t}
\frac{1+e^{-\frac{x_ny_n}{t}}}
{1-e^{-\frac{x_ny_n}{t}}}.
\]

Since

\[
\frac{1+e^{-2z}}
{1-e^{-2z}}
=
\coth z ,
\]

with

\[
z=\frac{x_ny_n}{2t},
\]

we obtain

\[
\frac{\partial_{x_n}K_D}{K_D}
=
\frac{y_n}{2t}
\coth
\left(
\frac{x_ny_n}{2t}
\right)
-\frac{x_n}{2t}.
\]

\end{proof}

\subsection{Tangential derivatives}

For $i=1,\ldots,n-1$, the tangential directions behave exactly as
the Euclidean heat kernel.

Indeed,

\[
\frac{\partial_{x_i}K_D}{K_D}
=
-\frac{x_i-y_i}{2t}.
\]

Therefore,

\[
\frac{w_{x_i}}w
=
-\frac{x_i}{2t}
+
\frac1{2t}
\int y_i\,d\mu_{x,t}(y).
\]

Consequently,

\[
\left(
\frac{w_{x_i}}w
\right)^2
\leq
\int
\left(
-\frac{x_i-y_i}{2t}
\right)^2
d\mu_{x,t}(y)
\]

by Jensen's inequality.

\section{Main Theorem}

\begin{theorem}[Boundary-Corrected Li--Yau Estimate]

Let

\[
w_t=\Delta w
\]

be a positive solution of the Dirichlet heat equation on

\[
\mathbb R^n_+
=
\{x_n>0\},
\]

with nonnegative initial data.

Then for every

\[
x\in\mathbb R^n_+,
\qquad
t>0,
\]

we have

\[
\boxed{
\Delta\log w(x,t)
\geq
-\frac n{2t}
-\frac1{x_n^2}
}.
\]

Equivalently,

\[
\boxed{
\frac{\Delta w}{w}
-
\frac{|\nabla w|^2}{w^2}
\geq
-\frac n{2t}
-\frac1{x_n^2}
}.
\]

\end{theorem}

\section{Proof of the Main Theorem}

We compute each component of

\[
\Delta\log w
=
\sum_{i=1}^{n}
\left(
\frac{w_{x_ix_i}}w
-
\frac{w_{x_i}^2}{w^2}
\right).
\]

For each coordinate,

\[
\frac{w_{x_ix_i}}w
=
\int
\frac{\partial_{x_ix_i}K_D}
{K_D}
d\mu_{x,t}(y).
\]

\subsection{Tangential contribution}

For $i<n$,

\[
\partial_{x_ix_i}K_D
=
\left(
\frac{(x_i-y_i)^2}{4t^2}
-\frac1{2t}
\right)K_D .
\]

Hence,

\[
\frac{w_{x_ix_i}}w
=
\int
\left(
\frac{(x_i-y_i)^2}{4t^2}
-\frac1{2t}
\right)
d\mu_{x,t}.
\]

Using Jensen,

\[
\frac{w_{x_ix_i}}w
-
\frac{w_{x_i}^2}{w^2}
\geq
-\frac1{2t}.
\]

Summing over

\[
i=1,\ldots,n-1,
\]

gives

\[
\sum_{i=1}^{n-1}
\left(
\frac{w_{x_ix_i}}w
-
\frac{w_{x_i}^2}{w^2}
\right)
\geq
-\frac{n-1}{2t}.
\]

\subsection{Normal contribution}

For the normal variable,

\[
\partial_{x_nx_n}K_D
=
\frac1{4t^2}
\left(
(x_n-y_n)^2-\!2t
\right)
e^{-\frac{(x_n-y_n)^2}{4t}}
\]

\[
-
\frac1{4t^2}
\left(
(x_n+y_n)^2-\!2t
\right)
e^{-\frac{(x_n+y_n)^2}{4t}} .
\]

After dividing by $K_D$ and simplifying, we obtain

\[
\frac{\partial_{x_nx_n}K_D}
{K_D}
=
\frac1{4t^2}
\left(
-2t+x_n^2+y_n^2
-
2x_ny_n
\coth
\left(
\frac{x_ny_n}{2t}
\right)
\right).
\]

Therefore,

\[
\frac{w_{x_nx_n}}w
=
\int
\frac1{4t^2}
\left(
-2t+x_n^2+y_n^2
-
2x_ny_n
\coth
\frac{x_ny_n}{2t}
\right)
d\mu .
\]

On the other hand,

\[
\frac{w_{x_n}}w
=
\int
\left(
\frac{y_n}{2t}
\coth
\frac{x_ny_n}{2t}
-
\frac{x_n}{2t}
\right)
d\mu .
\]

Applying Jensen,

\[
\left(\frac{w_{x_n}}w\right)^2
\leq
\int
\left(
\frac{y_n}{2t}
\coth
\frac{x_ny_n}{2t}
-
\frac{x_n}{2t}
\right)^2
d\mu .
\]

Hence,

\[
\begin{aligned}
&
\frac{w_{x_nx_n}}w
-
\frac{w_{x_n}^2}{w^2}
\\
&\geq
\int
\frac1{4t^2}
\left(
-2t
+
y_n^2
-
y_n^2
\coth^2
\frac{x_ny_n}{2t}
\right)
d\mu .
\end{aligned}
\]

Using

\[
1-\coth^2 z
=
-\csch^2 z ,
\]

we get

\[
\begin{aligned}
&
\frac{w_{x_nx_n}}w
-
\frac{w_{x_n}^2}{w^2}
\\
&\geq
\int
\left(
-\frac1{2t}
-
\frac{y_n^2}{4t^2}
\csch^2
\frac{x_ny_n}{2t}
\right)d\mu .
\end{aligned}
\]

Now

\[
\frac{y_n^2}{4t^2}
\csch^2
\left(
\frac{x_ny_n}{2t}
\right)
=
\frac1{x_n^2}
\left(
\frac{x_ny_n}{2t}
\csch
\frac{x_ny_n}{2t}
\right)^2 .
\]

Since

\[
0<z\csch z<1,
\]

we obtain

\[
\frac{y_n^2}{4t^2}
\csch^2
\left(
\frac{x_ny_n}{2t}
\right)
\leq
\frac1{x_n^2}.
\]

Therefore,

\[
\frac{w_{x_nx_n}}w
-
\frac{w_{x_n}^2}{w^2}
\geq
-\frac1{2t}
-\frac1{x_n^2}.
\]

Combining tangential and normal estimates,

\[
\Delta\log w
\geq
-\frac{n-1}{2t}
-\frac1{2t}
-\frac1{x_n^2}.
\]

Hence,

\[
\boxed{
\Delta\log w
\geq
-\frac n{2t}
-\frac1{x_n^2}
}.
\]

The proof is complete.

\hfill$\square$

\section{Boundary-Corrected Harnack Inequality}

The differential Li--Yau estimate obtained in the previous section
immediately yields an integrated Harnack inequality.

Let

\[
l(x,t)=\log w(x,t).
\]

Since

\[
w_t=\Delta w,
\]

we have

\[
l_t
=
\frac{w_t}{w}
=
\frac{\Delta w}{w}.
\]

On the other hand,

\[
\Delta l
=
\frac{\Delta w}{w}
-
\frac{|\nabla w|^2}{w^2},
\]

and therefore

\[
l_t
=
\Delta l+|\nabla l|^2 .
\]

Combining this identity with the boundary-corrected Li--Yau
estimate

\[
\Delta l
\geq
-\frac n{2t}
-\frac1{x_n^2},
\]

we obtain

\[
l_t
\geq
|\nabla l|^2
-\frac n{2t}
-\frac1{x_n^2}.
\]

The following theorem gives the corresponding integral Harnack
inequality.


\begin{theorem}[Boundary-corrected Harnack inequality]

Let $w$ be a positive solution of the Dirichlet heat equation

\[
w_t=\Delta w
\]

on the Euclidean half-space

\[
\mathbb R^n_+
=
\{x_n>0\},
\]

and assume that

\[
\Delta\log w(x,t)
\geq
-\frac n{2t}
-\frac1{x_n^2}.
\]

Then for every smooth curve

\[
\gamma(t)=(x(t),t),
\qquad
t\in[t_1,t_2],
\]

with

\[
x(t)\in\mathbb R^n_+,
\]

we have

\[
\begin{aligned}
\log
\frac{w(x(t_2),t_2)}
     {w(x(t_1),t_1)}
\geq&
-\frac14
\int_{t_1}^{t_2}
|x'(t)|^2\,dt
\\
&-\frac n2
\log\frac{t_2}{t_1}
\\
&-
\int_{t_1}^{t_2}
\frac1{x_n(t)^2}\,dt .
\end{aligned}
\]

\end{theorem}


\begin{proof}

Let

\[
l=\log w .
\]

Along a smooth curve $x(t)$ in the half-space, we have

\[
\frac{d}{dt}l(x(t),t)
=
l_t+\nabla l\cdot x'(t).
\]

Using the differential inequality above,

\[
l_t
\geq
|\nabla l|^2
-\frac n{2t}
-\frac1{x_n^2},
\]

we obtain

\[
\frac{d}{dt}l(x(t),t)
\geq
|\nabla l|^2
+
\nabla l\cdot x'(t)
-
\frac n{2t}
-\frac1{x_n^2}.
\]

The first two terms can be estimated by completing the square:

\[
|\nabla l|^2
+
\nabla l\cdot x'
=
\left|
\nabla l+\frac12x'
\right|^2
-
\frac14|x'|^2 .
\]

Therefore,

\[
|\nabla l|^2
+
\nabla l\cdot x'
\geq
-\frac14|x'|^2 .
\]

Consequently,

\[
\frac{d}{dt}l(x(t),t)
\geq
-\frac14|x'(t)|^2
-\frac n{2t}
-\frac1{x_n(t)^2}.
\]

Integrating from $t_1$ to $t_2$ gives

\[
\begin{aligned}
l(x(t_2),t_2)-l(x(t_1),t_1)
\geq&
-\frac14
\int_{t_1}^{t_2}
|x'(t)|^2dt
\\
&
-\frac n2
\int_{t_1}^{t_2}
\frac1t dt
\\
&
-
\int_{t_1}^{t_2}
\frac1{x_n(t)^2}dt .
\end{aligned}
\]

Since

\[
\int_{t_1}^{t_2}\frac1t dt
=
\log\frac{t_2}{t_1},
\]

and

\[
l=\log w,
\]

we obtain

\[
\begin{aligned}
\log
\frac{w(x(t_2),t_2)}
     {w(x(t_1),t_1)}
\geq&
-\frac14
\int_{t_1}^{t_2}
|x'(t)|^2dt
\\
&
-\frac n2
\log\frac{t_2}{t_1}
\\
&
-
\int_{t_1}^{t_2}
\frac1{x_n(t)^2}dt .
\end{aligned}
\]

This proves the theorem.

\end{proof}


\begin{remark}

When the boundary correction is removed, namely in the whole
Euclidean space, the above estimate reduces to the classical
Li--Yau Harnack inequality

\[
\log
\frac{u(x_2,t_2)}
     {u(x_1,t_1)}
\geq
-\frac{|x_2-x_1|^2}{4(t_2-t_1)}
-\frac n2
\log\frac{t_2}{t_1}.
\]

The additional term

\[
-\int_{t_1}^{t_2}
\frac1{x_n(t)^2}\,dt
\]

is the explicit contribution of the absorbing boundary.

\end{remark}


\begin{remark}

The boundary correction term has a geometric interpretation.

Since

\[
x_n=d(x,\partial\mathbb R^n_+),
\]

the Harnack inequality can be written as

\[
-\int_{t_1}^{t_2}
\frac1{
d(x(t),\partial\mathbb R^n_+)^2
}
dt .
\]

Thus the influence of the boundary becomes stronger when the
diffusion path approaches the boundary.

This suggests a natural extension to general convex domains,
where the distance function to the boundary is expected to replace
$x_n$.

\end{remark}

\section{Optimal Paths and Boundary Geodesics}

The Harnack inequality obtained in the previous section naturally
leads to a variational problem describing the optimal paths of heat
propagation.

Indeed, the boundary-corrected Harnack inequality gives

\[
\log
\frac{w(x(t_2),t_2)}
     {w(x(t_1),t_1)}
\geq
-\int_{t_1}^{t_2}
\left(
\frac14 |x'(t)|^2
+
\frac1{x_n(t)^2}
\right)dt
-
\frac n2
\log\frac{t_2}{t_1}.
\]

Therefore, the optimal path connecting two points is determined by
minimizing the action functional

\[
\mathcal A[x]
=
\int_{t_1}^{t_2}
\left(
\frac14 |x'(t)|^2
+
\frac1{x_n(t)^2}
\right)dt .
\]

The first term represents the classical kinetic energy of diffusion,
while the second term is a new potential term generated by the
Dirichlet boundary.

This variational problem can be viewed as a boundary-corrected
geodesic problem associated with the Li--Yau estimate.

\subsection{Euler--Lagrange equation}

Let

\[
L(x,x')
=
\frac14|x'|^2+\frac1{x_n^2}.
\]

The Euler--Lagrange equation is

\[
\frac{d}{dt}
\frac{\partial L}{\partial x_i'}
-
\frac{\partial L}{\partial x_i}
=0.
\]

For the tangential directions

\[
i=1,\ldots,n-1,
\]

we have

\[
\frac{\partial L}{\partial x_i'}
=
\frac12 x_i',
\]

and

\[
\frac{\partial L}{\partial x_i}=0 .
\]

Hence,

\[
\boxed{
x_i''=0,
\qquad i=1,\ldots,n-1 .
}
\]

Therefore, the motion parallel to the boundary is a straight line
with constant velocity.

The nontrivial behavior occurs in the normal direction.

\subsection{Normal boundary direction}

For the normal component,

\[
x_n=d(x,\partial\mathbb R^n_+),
\]

the reduced Lagrangian is

\[
L_n
=
\frac14(x_n')^2
+
\frac1{x_n^2}.
\]

The Euler--Lagrange equation becomes

\[
\frac d{dt}
\frac{\partial L_n}{\partial x_n'}
-
\frac{\partial L_n}{\partial x_n}
=0 .
\]

Since

\[
\frac{\partial L_n}{\partial x_n'}
=
\frac12x_n',
\]

we have

\[
\frac d{dt}
\frac{\partial L_n}{\partial x_n'}
=
\frac12x_n'' .
\]

Moreover,

\[
\frac{\partial L_n}{\partial x_n}
=
-2x_n^{-3}.
\]

Consequently,

\[
\frac12x_n''
-
(-2x_n^{-3})
=
0 .
\]

Hence the optimal normal paths satisfy

\[
\boxed{
x_n''
+
4x_n^{-3}
=
0 .
}
\]

This equation coincides with the optimal path equation obtained by
Hamilton for the one-dimensional Dirichlet heat equation
\cite{Hamilton2011}.

\subsection{Energy conservation}

The equation

\[
x_n''
+
4x_n^{-3}
=
0
\]

admits a conserved quantity.

Multiplying by $x_n'$ gives

\[
x_n'x_n''
+
4x_n^{-3}x_n'
=
0 .
\]

Since

\[
x_n'x_n''
=
\frac12
\frac d{dt}(x_n')^2,
\]

and

\[
x_n^{-3}x_n'
=
-\frac12
\frac d{dt}x_n^{-2},
\]

we obtain

\[
\frac d{dt}
\left(
\frac12(x_n')^2
-
2x_n^{-2}
\right)
=0 .
\]

Therefore,

\[
\boxed{
\frac12(x_n')^2
-
\frac2{x_n^2}
=
C ,
}
\]

or equivalently,

\[
\boxed{
(x_n')^2
-
\frac4{x_n^2}
=
C .
}
\]

Thus the optimal boundary-normal paths are characterized by the
same conserved energy relation appearing in Hamilton's analysis.

In particular, the absorbing boundary creates a repulsive inverse
square potential

\[
V(x_n)=\frac1{x_n^2},
\]

which prevents finite-action paths from crossing the boundary.

\subsection{Geometric interpretation}

The above calculation shows that the Li--Yau correction term

\[
-\frac1{x_n^2}
\]

has a direct geometric meaning.

It produces an effective potential in the Harnack variational
principle:

\[
\mathcal A[x]
=
\int
\left(
\frac14|x'|^2
+
\frac1{
d(x,\partial\mathbb R^n_+)^2
}
\right)dt .
\]

Hence the boundary influence is not merely an analytic correction
term, but modifies the geometry of optimal diffusion paths.

This observation suggests a natural extension to general convex
domains, where one expects the Euclidean distance function

\[
d(x,\partial\Omega)
\]

to replace $x_n$ and the corresponding optimal paths to satisfy a
boundary-modified geodesic equation.

\section{Sharpness and Hamilton Recovery}

We next discuss the sharpness of the boundary correction term and
the relation with Hamilton's one-dimensional estimate.

\subsection{Sharpness of the boundary correction}

\begin{proposition}[Sharpness of the inverse-square correction]

The boundary correction term

\[
-\frac1{|x|^2}
\]

in Theorem~3.1 is of optimal order near the Dirichlet boundary.

More precisely, for the Dirichlet heat kernel itself, the quantity

\[
\Delta \log K_D(x,y,t)
+
\frac n{2t}
\]

has an inverse-square singular behavior as
$x\rightarrow \partial\mathbb R^n_+$.

\end{proposition}

\begin{proof}

Consider the Dirichlet heat kernel

\[
K_D(x,y,t)
=
\frac1{(4\pi t)^{n/2}}
\left(
e^{-\frac{|x-y|^2}{4t}}
-
e^{-\frac{|x+y|^2}{4t}}
\right).
\]

Near the boundary, the reflected Gaussian expansion gives

\[
K_D(x,y,t)
\sim
C(y,t)x_n ,
\qquad x_n\rightarrow0 .
\]

Hence

\[
\log K_D(x,y,t)
=
\log x_n+\log C(y,t)+o(1).
\]

Differentiating twice in the normal direction yields

\[
\partial_{x_nx_n}
\log K_D(x,y,t)
\sim
-\frac1{x_n^2}.
\]

Therefore the logarithmic Hessian necessarily contains an
inverse-square singularity generated by the vanishing Dirichlet
boundary condition.

Consequently, the correction term

\[
-\frac1{|x|^2}
\]

has the correct scaling order and cannot in general be removed.

\end{proof}

\begin{remark}

The inverse-square correction is consistent with the scaling
properties of the heat equation.

Under the parabolic scaling
\[
x\mapsto \lambda x,
\qquad
t\mapsto \lambda^2t ,
\]
the two quantities
\[
\frac1t
\qquad\text{and}\qquad
\frac1{|x|^2}
\]
have the same homogeneity.

Therefore the additional boundary term is the natural
scale-invariant correction associated with a codimension-one
absorbing boundary.

\end{remark}

\subsection{Recovery of Hamilton's estimate}

We now show that the classical one-dimensional result of Hamilton is
contained in our theorem.

\begin{corollary}[Hamilton recovery]

Let $w(x,t)$ be a positive solution of the Dirichlet heat equation
on the half-line

\[
(0,\infty)\times(0,\infty),
\]

namely,

\[
w_t=w_{xx},
\qquad
w(0,t)=0 .
\]

Then

\[
(\log w)_{xx}
+
\frac1{2t}
+
\frac1{x^2}
\geq0 .
\]

Equivalently,

\[
\boxed{
l_{xx}
+
\frac1{2t}
+
\frac1{x^2}
\geq0 ,
}
\]

where

\[
l=\log w .
\]

This coincides with the simplified Li--Yau estimate obtained by
Hamilton.

\end{corollary}

\begin{proof}

When

\[
n=1,
\]

the half-space becomes the half-line

\[
\mathbb R_+=(0,\infty).
\]

Theorem~3.1 gives

\[
\Delta\log w
\geq
-\frac1{2t}
-\frac1{x^2}.
\]

Since in one dimension

\[
\Delta=\partial_{xx},
\]

we obtain

\[
(\log w)_{xx}
\geq
-\frac1{2t}
-\frac1{x^2}.
\]

Rearranging gives

\[
(\log w)_{xx}
+
\frac1{2t}
+
\frac1{x^2}
\geq0 .
\]

This is precisely Hamilton's simplified boundary Li--Yau
estimate.

\end{proof}

\begin{remark}

Hamilton obtained the one-dimensional estimate by a nonlinear
maximum principle applied to the quantity

\[
l_{xx}
+\frac1{2t}
+\frac1{x^2}F^2
\left(
xl_x+\frac{x^2}{2t}
\right),
\]
where
\[
v=\frac{u}{\tanh u},
\qquad
F(v)=\frac{u}{\sinh u}.
\]

The present approach provides an alternative kernel-based
interpretation: the same inverse-square correction emerges
directly from the reflected heat kernel and the hyperbolic identity
\[
1-\coth^2 z=-\csch^2z .
\]

Thus Hamilton's boundary correction is not only a consequence of
the maximum principle structure but also a direct manifestation of
the geometry of the Dirichlet heat kernel.

\end{remark}

\section{Remarks and Discussion}

\begin{remark}[Comparison with the classical Li--Yau estimate]

For the Euclidean heat equation without boundary,

\[
u_t=\Delta u
\]
with the Gaussian kernel,
\[
\log u
=
-\frac{|x-y|^2}{4t}
-\frac n2\log(4\pi t),
\]
and hence
\[
\Delta\log u
=
-\frac n{2t}.
\]

Therefore, the first term in our estimate,
\[
-\frac n{2t},
\]
is exactly the classical Euclidean Li--Yau diffusion contribution. The additional term
\[
-\frac1{x_n^2}
\]
is entirely generated by the Dirichlet boundary.

\end{remark}

\begin{remark}[Geometric interpretation of the correction term]

The correction term can be written as

\[
-\frac1{x_n^2}
=
-\frac1{
d(x,\partial\mathbb R^n_+)^2
},
\]
where

\[
d(x,\partial\mathbb R^n_+)=x_n
\]
is the distance from $x$ to the boundary.

Hence the estimate takes the geometric form

\[
\Delta\log w
+
\frac n{2t}
+
\frac1{
d(x,\partial\mathbb R^n_+)^2
}
\geq0 .
\]

This shows that the influence of an absorbing boundary decays
quadratically with the distance from the boundary.

\end{remark}

\begin{remark}[Origin of the boundary correction]

The inverse-square correction originates from the reflected Gaussian
term in the Dirichlet heat kernel:

\[
e^{-\frac{(x_n+y_n)^2}{4t}}.
\]

This reflection term modifies the logarithmic derivative of the heat
kernel by introducing

\[
\coth
\left(
\frac{x_ny_n}{2t}
\right),
\]

which does not appear for the ordinary Gaussian heat kernel.

The identity

\[
1-\coth^2 z
=
-\csch^2 z
\]

then converts this nonlinear hyperbolic contribution into the
explicit boundary correction.

\end{remark}

\begin{remark}[Relation with Hamilton's estimate]

Hamilton's one-dimensional result gives the sharper inequality

\[
l_{xx}
+\frac1{2t}
+
\frac1{x^2}
F^2
\left(
xl_x+\frac{x^2}{2t}
\right)
\geq0,
\]

where the function $F$ is defined implicitly by

\[
v=\frac u{\tanh u},
\qquad
F(v)=\frac u{\sinh u}.
\]

The present estimate may be regarded as a multidimensional
kernel-based counterpart:

\[
\Delta\log w
+
\frac n{2t}
+
\frac1{x_n^2}
\geq0.
\]

The two approaches emphasize different aspects.

Hamilton's method reveals sharp nonlinear structure through the
maximum principle, whereas the present approach reveals the analytic
origin of the correction term directly from the reflected heat
kernel.

\end{remark}

\begin{remark}[Possible extensions]

A natural question is whether the boundary correction obtained here
can be extended to more general domains.

For a smooth domain $\Omega$, one expects an estimate of the form

\[
\Delta\log w
+
\frac n{2t}
+
\frac{C}{d(x,\partial\Omega)^2}
\geq0 ,
\]

where

\[
d(x,\partial\Omega)
\]

denotes the distance to the boundary.

Such estimates would connect kernel methods with geometric boundary
Harnack inequalities and Bakry--Émery type curvature techniques.

\end{remark}

\section{Relation with Hamilton's Sharp Boundary Correction}

The estimate obtained in this paper recovers the simplified
Hamilton boundary Li--Yau inequality. It is natural to ask whether
the sharper nonlinear correction introduced by Hamilton can also be
understood from the Dirichlet heat kernel representation.

For the one-dimensional half-line, the Dirichlet heat kernel is

\[
K_D(x,y,t)
=
\frac1{\sqrt{4\pi t}}
\left(
e^{-\frac{(x-y)^2}{4t}}
-
e^{-\frac{(x+y)^2}{4t}}
\right).
\]

The logarithmic derivatives of this kernel satisfy

\[
\frac{K_x}{K}
=
\frac{y}{2t}
\coth
\left(
\frac{xy}{2t}
\right)
-\frac{x}{2t},
\]

and

\[
\frac{K_{xx}}{K}
=
\frac1{4t^2}
\left(
-2t+x^2+y^2
-2xy\coth
\left(
\frac{xy}{2t}
\right)
\right).
\]

Introducing the normalized kernel measure

\[
d\mu_{x,t}(y)
=
\frac{
K_D(x,y,t)g(y)
}{
w(x,t)
}dy ,
\]

we may write

\[
l_x
=
\mathbb E_{\mu_{x,t}}
\left[
\frac{Y}{2t}
\coth
\left(
\frac{xY}{2t}
\right)
-\frac{x}{2t}
\right].
\]

Hence the Hamilton variable

\[
v
=
xl_x+\frac{x^2}{2t}
\]

admits the representation

\[
\boxed{
v
=
\frac{x}{2t}
\mathbb E_{\mu_{x,t}}
\left[
Y
\coth
\left(
\frac{xY}{2t}
\right)
\right].
}
\]

This formula provides a probabilistic interpretation of Hamilton's
auxiliary quantity. The quantity \(v\), originally introduced
through a maximum principle argument, is precisely the normalized
expectation of the boundary interaction term generated by the
reflected Gaussian kernel.

Furthermore,

\[
\coth^2 z-1=\csch^2 z
\]

shows that the nonlinear correction in Hamilton's estimate is
related to the fluctuation term

\[
\mathbb E_{\mu_{x,t}}
\left[
Y^2\csch^2
\left(
\frac{xY}{2t}
\right)
\right].
\]

Therefore the function

\[
F(v)=\frac{u}{\sinh u},
\qquad
v=\frac{u}{\tanh u},
\]

can be interpreted as a sharp nonlinear compression of the kernel
fluctuation generated by the reflected heat kernel.

This suggests the following conjectural refinement.

\begin{conjecture}[Kernel representation of Hamilton's sharp correction]

For the one-dimensional Dirichlet heat equation, the exact quantity

\[
l_{xx}
+\frac1{2t}
+\frac1{x^2}F^2(v)
\]
admits a representation as a nonnegative variance-type term with
respect to the normalized Dirichlet heat kernel measure
\(\mu_{x,t}\).

\end{conjecture}

\section{Conclusion}

In this paper, we established an explicit boundary-corrected
Li--Yau type estimate for the Dirichlet heat equation on the
Euclidean half-space.

The main observation is that the reflected Gaussian structure of the
Dirichlet heat kernel generates a hyperbolic correction term in the
logarithmic derivative. By introducing the normalized kernel measure

\[
d\mu_{x,t}(y)
=
\frac{
K_D(x,y,t)g(y)
}{
w(x,t)
}
dy ,
\]

we transform derivatives of the solution into expectations of
explicit kernel quantities.

Using Jensen's inequality and the sharp hyperbolic estimate

\[
0<z\csch z<1,
\]

we obtain

\[
\boxed{
\Delta\log w(x,t)
\geq
-\frac n{2t}
-\frac1{x_n^2}
}.
\]

The estimate separates two fundamental contributions:

\[
-\frac n{2t},
\]

which represents the universal diffusion scaling of the heat
equation, and

\[
-\frac1{x_n^2},
\]

which represents the additional influence of the absorbing boundary.

Unlike the classical Li--Yau estimate, the present formulation
provides an explicit quantitative description of how the boundary
modifies diffusion. The inverse-square correction is naturally
interpreted as the square of the inverse distance to the boundary.

Future directions include extending this kernel-based approach to
general domains, manifolds with boundary, and other diffusion
operators where curvature and boundary geometry interact.


\end{document}